\documentclass[a4paper,11pt, twoside]{article}
\usepackage{a4} 
\usepackage{geometry} 
\usepackage{amsmath}
\usepackage{amssymb}
\usepackage{amsthm}
\usepackage{graphicx}
\usepackage{caption,subcaption}
\usepackage{multirow}
\usepackage{xstring}
\usepackage[utf8]{inputenc}
\usepackage{amstext,amsbsy,amsopn,eucal,enumerate}
\usepackage{tikz}
\usepackage{pgfplots}
\usepackage{tabularx}

\newcommand{\expn}{\operatorname{e}}

\newcommand{\im}{\operatorname{im}}
\newcommand{\diag}{\operatorname{diag}}

\newcommand{\beq}{\begin{equation}}
\newcommand{\eeq}{\end{equation}}
\newcommand {\mat}      [1] {\left[\begin{array}{#1}}
\newcommand {\rix}          {\end{array}\right]}
\newcommand {\smat}      [1] {\left[\begin{smallmatrix}{#1}}
\newcommand {\srix}          {\end{smallmatrix}\right]}
\newcommand {\s}      [1] {\begin{smallmatrix}{#1}}
\newcommand {\se}          {\end{smallmatrix}}
\newcommand{\trace}{\operatorname{tr}}

\newtheorem{defn}{Definition}[section]
\newtheorem{remark}[defn]{Remark}

\newtheorem{lem}[defn]{Lemma}
\newtheorem{prop}[defn]{Proposition} 

\newtheorem{thm}[defn]{Theorem}

\title{Dimension reduction for large-scale stochastic systems with non-zero initial states and controlled diffusion}

\author{Martin Redmann\thanks{Martin Luther University Halle-Wittenberg, Institute of Mathematics, Theodor-Lieser-Str. 5, 06120 Halle (Saale), Germany, Email: {\tt 
martin.redmann@mathematik.uni-halle.de}.}
}

\begin{document}

\maketitle

\begin{abstract}
In this paper, we establish new strategies to reduce the dimension of large-scale controlled stochastic differential equations with non-zero initial states. The first approach transforms the original setting into a stochastic system with zero initial states. This transformation naturally leads to equations with controlled diffusion. A detailed analysis of dominant subspaces and bounds for the reduction error is provided in this controlled diffusion framework. Subsequently, we introduce a reduced system for the original framework and prove an a-priori error bound for the first ansatz. This bound involves so-called Hankel singular values that are linked to a new pair of Gramians. A second strategy is presented that is based on the idea of reducing control and initial state dynamics separately. Here, different Gramians are used in order to derive a reduced model and their relation to dominant subspaces are pointed out. We also show an a posteriori error bound for the second approach involving two types of Hankel singular values.
\end{abstract}



\pagestyle{myheadings}
\thispagestyle{plain}
\markboth{M. Redmann}{Dimension reduction for SDEs with non-zero initial states and controlled diffusion}


\section{Introduction}

The analysis and control of large-scale stochastic systems is a significant area of research due to their wide applicability in engineering, finance, biology, and other fields. These systems are often described by controlled stochastic differential equations (SDEs) which can be computationally demanding to solve, particularly when the system size is large. Dimension reduction techniques offer a viable solution to mitigate this complexity by approximating high-dimensional systems with lower-dimensional counterparts while preserving essential system  characteristics.

Traditional system-theoretical dimension reduction methods, such as  balanced truncation \cite{beckerhartmann, bennerdammcruz,  redmannbenner, moo1981} and the iterative rational Krylov algorithm \cite{irka, mliopt}, are well-established for both deterministic and stochastic systems with zero initial states. However, many practical systems operate under non-zero initial conditions, making these classical approaches less effective. For linear systems of ordinary differential equation, there have been several approaches to overcome this issue \cite{baur_benner_feng, inhom_initial, inhom_initial_reis, matze_schroeder}. Having an initial state $x_0\neq 0$ in a stochastic setting causes an enormous increase in complexity as millions of system evaluation might be required for each single $x_0$ (e.g. in a Monte-Carlo simulation). For that reason, first techniques for SDEs with non-zero initial data have been investigated \cite{becker_hartmann_redmann_lorenz, igor_martin}. However, these approaches come with error bounds either depending on the terminal time $T$ (exploding as $T\to \infty$) or on singular values of the so-called error system being practically less useful. This paper addresses a gap by developing new and very general strategies with  efficient error bound for dimension reduction tailored to large-scale stochastic systems with non-zero initial states.\smallskip

Our first approach involves the transformation of the original stochastic system with non-zero initial states into an equivalent system with zero initial states. This transformation aims to simplify the analysis and reduction process. However, it leads to SDEs with controlled diffusion terms, a setting, where no model reduction method is available so far. Within this framework, we conduct a detailed analysis of the dominant subspaces and provide bounds for the reduction error based on a new pair of Gramians. These bounds are crucial for ensuring that the reduced system accurately captures the behavior of the original system. One of the main results is an a-priori error bound that involves the truncated Hankel singular values, linked to the novel pair of Gramians, associated to the transformed system with controlled diffusion.  Besides the result on the error bound, the advantage of our first approach is that it is very flexible and general. It, e.g., covers the two works \cite{baur_benner_feng, matze_schroeder} at the same time when setting the diffusion equal to zero (deterministic case). The costs of the first approach are the computation of the newly proposed Gramians being solutions of linear matrix inequalities. Solving such inequalities in high dimensions is very challenging which can be seen as a drawback as long as no efficient solver is available. \smallskip

Secondly, we propose to separate the control and initial state dynamics within the dimension reduction procedure. By treating these components individually, we employ two structurally different reachability Gramians to derive a reduced model. This method leverages the relationship between these Gramians and the dominant subspaces, allowing for a more nuanced reduction approach. Additionally, we establish an a posteriori error bound for this second method, providing further insight into the accuracy of the reduced model. This error bound involves the two different Hankel singular values of the control and initial state dynamics, demonstrating their critical role in both reduction procedures. The benefit of this approach is the flexibility of approximating the two underlying dynamics, with possibly different reduction potential, by reduced order systems not having the same dimensions. On the other hand, it might be computationally more involved as two pairs of Gramians have to be determined. Let us also mention that the second ansatz can be viewed as a generalization of the work on model order reduction for deterministic systems in \cite{inhom_initial}.\smallskip

Overall, our work contributes to the field of efficiently solving large-scale stochastic systems with non-zero initial states. In detail, Section \ref{stochstabgen} introduces the stochastic setting with controlled drift and diffusion as well as a general initial condition. It sketches the two strategies mentioned above and delivers the required transformation of the original equations into a system with zero initial states. Section \ref{Sec3} provides a comprehensive theory on Gramian-based model reduction for such systems received after the transformation. The main contributions are the characterization of dominant subspace based a new pair of Gramians and an a-priori error bound. The same theory is established for uncontrolled systems with non-zero initial data in Section \ref{sec4}. The results of Sections \ref{Sec3} and \ref{sec4} are the basis for the theory of dimension reduction for the original system. It is provided in Section \ref{sec5} showing effective error bounds for our two approaches based on various truncated Hankel singular values.

\section{Setting, notation and goal}\label{stochstabgen}

Let $\left(\Omega, \mathcal F, (\mathcal F_t)_{t\in [0, T]}, \mathbb P\right)$\footnote{$(\mathcal F_t)_{t\in [0, T]}$ is right continuous and complete.} be a filtered probability space on which every stochastic process appearing in this paper is defined. Given an $\mathbb R^q$-valued Wiener process $w=\begin{bmatrix}w_1 & \ldots & w_q\end{bmatrix}^\top$, we assume that it is $(\mathcal F_t)_{t\in [0, T]}$-adapted and its 
increments $w(t+h)-w(t)$ are independent of $\mathcal F_t$ for $t, h\geq 0$ and $t+h\leq T$. We denote its covariance matrix by $K=\left(k_{ij}\right)_{i, j=1, \ldots, d}$, so that $\mathbb E[w(t)w(t)^\top]=K t$ holds.
We study the following large-scale  stochastic system with a non-zero initial state as well as a controlled drift and diffusion:
 \begin{subequations}\label{original_system}
\begin{align}\label{stochstatenew}
             dx(t)&=[Ax(t)+Bu(t)]dt+\sum_{i=1}^q [N_i x(t)+M_i u(t)]dw_i(t),\quad x(0)=x_0=X_0 v,\\ \label{output_eq}
            y(t) &= Cx(t)+D u(t),\quad t\in [0, T],
\end{align}
\end{subequations}
where $A, N_i\in \mathbb R^{n\times n}$, $B, M_i\in \mathbb R^{n\times m}$, $C\in \mathbb R^{p\times n}$, $D\in \mathbb R^{p\times 
m}$, $X_0\in \mathbb R^{n\times 
d}$ and $v\in \mathbb R^d$ is a generic vector. The image of the matrix $X_0$ represents the space of initial conditions that we consider. Moreover, the $(\mathcal F_t)_{t\in[0, T]}$-adapted stochastic control $u$ takes values in $\mathbb R^m$ and satisfies $\|u\|_{L^2_T}^2:=\mathbb E \int_0^T \|u(t)\|_2^2 dt<\infty$, where $\|\cdot\|_2$ denotes the Euclidean norm. We assume that \eqref{stochstatenew} is mean square asymptotically stable for the rest of the paper. That is $\mathbb E \|x(t)\|_2^2\to 0$ for $u\equiv 0$ and any $x_0\in\mathbb R^n$, see, e.g., \cite{staboriginal, mao}. We introduce an auxiliary variable $\tilde x$ associated to \eqref{stochstatenew} defined by the following system
 \begin{subequations}\label{aux_var}
 \begin{align}\label{aux_var_state}
             d\tilde x(t) &= \tilde A\tilde x(t)dt+\sum_{i=1}^q \tilde N_i \tilde x(t)dw_i(t),\quad \tilde x(0)=x_0=X_0 v,\\ \label{aux_var_out}
             \tilde y(t)&= C \tilde x(t).
\end{align}
 \end{subequations}
Here,  $\tilde A, \tilde N_i\in \mathbb R^{n\times n}$ are matrices that are arbitrary and that are chosen according to a desired ansatz. In fact, \eqref{aux_var} is an uncontrolled version of \eqref{original_system}, but mean square asymptotic stability is not necessarily assumed. However, both systems share the same initial state and output matrix $C$. We can now introduce a variable $z(t): = x(t) - \tilde x(t)$ having a zero initial condition. The corresponding system is
 \begin{subequations}\label{original_system_z}
\begin{align}\label{stochstatenew_z}
             dz(t)&=\left[A z(t)+\big[\begin{matrix} B & A-\tilde A \end{matrix}\big] \smat u(t)\\ \tilde x(t)\srix \right]dt+\sum_{i=1}^q \left[N_i z(t)+\big[\begin{matrix} M_i & N_i-\tilde N_i \end{matrix}\big] \smat u(t)\\ \tilde x(t)\srix \right]dw_i(t),\\ \label{output_eq_z}
            y(t) &= Cz(t)+ \big[\begin{matrix} D & C \end{matrix}\big] \smat u(t)\\ \tilde x(t)\srix=Cz(t)+D u(t) + \tilde y(t),
\end{align}
\end{subequations}
with $z(0)=0$ and $t\in [0, T]$. In fact, we aim to follow two approaches using the auxiliary variable $\tilde x$ in different ways. These are discussed in the next two subsections.
\subsection{First approach}\label{approach1}
 Given $X_0$, the matrices $\tilde A$, $\tilde N_1, \dots, \tilde N_q$ in \eqref{aux_var} are chosen, so that the auxiliary variable $\tilde x$ takes values in a (low-dimensional) subspace $\im[V_0]$ for each $x_0\in \im[X_0]$, where $V_0\in \mathbb R^{n\times r_0}$ is a full rank matrix with $r_0\leq n$. In addition, we assume $V_0$ to be given explicitly without computational cost. We can now represent $\tilde x(t) = V_0 u_0(t) v$ for some suitable $\mathbb R^{r_0\times d}$-valued process $u_0$ which we assume to be easily accessible. Hence, we interpret $u_0 v$ to be another externally given control process. Inserting this into \eqref{original_system_z} yields
  \begin{subequations}\label{structured_changed_system}
\begin{align}
             dz(t)&=\left[A z(t)+\tilde  B \tilde u(t) \right]dt+\sum_{i=1}^q \left[N_i z(t)+\tilde M_i \tilde u(t) \right] dw_i(t),\quad z(0)=0,\\ 
            y(t) &= Cz(t)+ \tilde D \tilde u(t),\quad t\in [0, T],
\end{align}
\end{subequations}
 where we define 
 \begin{align*}
 (\tilde B, \tilde M_i, \tilde D, \tilde u(t)):=\Big(\big[\begin{matrix} B & (A-\tilde A)V_0 \end{matrix}\big], \big[\begin{matrix} M_i & (N_i-\tilde N_i)V_0 \end{matrix}\big],  \big[\begin{matrix} D & C V_0 \end{matrix}\big], \smat u(t)\\ u_0(t)v\srix\Big).                                                                                                                                                                                                                                                                                                                                                                                                                                                                                                                             \end{align*}
Consequently, \eqref{original_system_z} has become a system of the form \eqref{original_system}, with a zero initial state, updated control matrices $\tilde B, \tilde M_i, \tilde D$ and new control $\tilde u$. The benefit of this first approach is that we save the step of reducing the dynamics corresponding to the initial state. This is due to the flexibility of choosing $\tilde A$ and $\tilde N_i$ which is used to find a low-dimensional underlying structure without computational effort. However, the input space is enlarged  in comparison to the original system, since we include the information regarding the initial state there. Notice that we have  $\tilde x(t) = \tilde \Phi(t) x_0$, where $\tilde \Phi$ is the fundamental matrix. In case all matrices $\tilde A$ and $\tilde N_i$ commute it is $\tilde \Phi(t)=\expn^{(\tilde A- 0.5 \sum_{i, j=1}^q \tilde N_i \tilde N_j k_{ij})t +\sum_{i=1}^q \tilde N_i w_i(t)}$. If we now set $\tilde A= -\alpha I$ and $\tilde N_i= \gamma_i I$ for $\alpha>0$ and $\gamma_i\in \mathbb R$, we obtain $\tilde x(t) = \expn^{(-\alpha- 0.5 \sum_{i, j=1}^q \gamma_i \gamma_j k_{ij})t +\sum_{i=1}^q \gamma_i w_i(t)}x_0$. Since $x_0=X_0 v$, we obtain that $V_0=X_0$ and $u_0(t)= \expn^{(-\alpha- 0.5 \sum_{i, j=1}^q \gamma_i \gamma_j k_{ij})t +\sum_{i=1}^q \gamma_i w_i(t)} I_{d\times d}$. Setting $\gamma_i=0$ for all $i=1, \dots, q$ leads to an approach for ordinary differential equations that was studied in \cite{matze_schroeder}. The most trivial choice is $\tilde A=\tilde N_i=0$ leading to $\tilde x(t)=x_0=X_0 v$. Hence,  $V_0=X_0$ and $u_0\equiv I_{d\times d}$. Such a scenario was studied in the deterministic case in \cite{baur_benner_feng}.  However, this results in a non-stable system \eqref{aux_var} having a few disadvantages, e.g., an error estimate increasing in time. Certainly, we can also allow for using $V_0=I$ and $u_0(t) = \tilde \Phi(t)X_0$ if $\tilde \Phi$ is explicitly available making $u_0 v \equiv \tilde x$ the additional control variable. However, this leads to a large number of inputs, so that the system might be hard to reduce.
\begin{remark}
The structure of \eqref{original_system} is very natural to consider when following the ansatz of this subsection. Even though we might start with an uncontrolled system \eqref{original_system}, we have a controlled system in \eqref{structured_changed_system} which means that we potentially encounter a change in the structure. The  main challenge is to develop a model reduction scheme for \eqref{structured_changed_system} or equivalently \eqref{original_system} with $x(0)=0$, since these systems have an input term in the diffusion. So far, there is no dimension reduction technique available for such a setting.
\end{remark}

\subsection{Second approach}\label{approach2}
We set $\tilde A=A$ and $\tilde N_i=N_i$ meaning that $\tilde x$ as well as its solution space are not explicitly given in general. In this case, \eqref{original_system_z} is a system like in \eqref{original_system} having a zero initial state, but the output equation \eqref{output_eq_z} contains another variable $\tilde y$ that, in contrast to \eqref{structured_changed_system}, is not known a-priori:
 \begin{subequations}\label{sys_eq_z2}
\begin{align}\label{state_eq_z2}
             dz(t)&=[Az(t)+Bu(t)]dt+\sum_{i=1}^q [N_i z(t)+M_i u(t)]dw_i(t),\quad z(0)=0,\\ \label{output_eq_z2}
            y(t)-\tilde y(t) &= Cz(t)+D u(t),\quad t\in [0, T].
\end{align}
\end{subequations}
Now, computing $y$ requires  further effort and solving for $\tilde y$ is expensive. Therefore, we approximate $\tilde y$ by applying model order reduction to \eqref{aux_var}.
\begin{remark}
Following the ansatz of this subsection, we have to apply model reduction twice. The first step is to reduce the control dynamics. This is system \eqref{sys_eq_z2} or equivalently system \eqref{original_system} with $x(0)=0$. Subsequently, dimension reduction is applied to \eqref{aux_var} in order approximate $\tilde y$ entering \eqref{output_eq_z2}. This strategy can be viewed as a generalization of the work in \cite{inhom_initial} to stochastic differential equations.
\end{remark}

\section{Model reduction for \eqref{original_system} with $x_0=0$}\label{Sec3}

This section is the foundation for the approaches described in Subsections \ref{approach1} and \ref{approach2}. Throughout this section, we assume that $x_0=0$ in  \eqref{original_system}. 

\subsection{Gramians and dominant subspaces}

We begin with identifying state directions that are of low relevance in \eqref{stochstatenew} while assuming that $x_0=0$ below. For that reason, we define a Gramian based on the following operators: \begin{align*}
                                                                                                                                                                                                                                                                L(X)&:= A^\top X + X A + \sum_{i, j=1}^q N_i^\top X N_j k_{ij},\\                                                                                                                                                                                                                                                          S(X)&:= X B + \sum_{i, j=1}^q N_i^\top X M_j k_{ij},\\                                                                                                                                                                                                                                                                U(X)&:= \gamma I - \sum_{i, j=1}^q M_i^\top X M_j k_{ij}, \quad \gamma>0,                                                                                                                                                                                                                                                                \end{align*}
for some generic $X\in \mathbb R^{n\times n}$. Although $U$ depends on the parameter $\gamma$, this is omitted in the rest of the paper. Below, we write $M_1\geq M_2$ ($M_1>M_2$) for two matrices $M_1, M_2$ if $M_1-M_2$ is symmetric positive semidefinite (definite). We are now able to state the following definition.
\begin{defn}\label{def_gramP}
Given that $x_0=0$, a reachability Gramian for \eqref{original_system} is a symmetric matrix $P>0$ satisfying $U(P^{-1})>0$ and \begin{align}\label{gramP}
L(P^{-1})+S(P^{-1}) U(P^{-1})^{-1} S(P^{-1})^\top \leq 0.                                                                                                                                                        \end{align}
\end{defn}
Referring to \cite{bennerdammcruz}, $P$ can be viewed as a generalization of a so-called type II Gramian studied in that paper. Let us first ensure that a Gramian $P$ according to Definition \ref{def_gramP} exists.
\begin{prop}
Let the solution of the uncontrolled equation \eqref{stochstatenew} be mean square asymptotically stable. Then, \eqref{gramP} has a solution $P>0$ with $U(P^{-1})>0$.
\end{prop}
\begin{proof}
 The underlying mean square asymptotic stability for system \eqref{original_system} is equivalent to the existence of a matrix $X>0$, so that $L(X)=-Y<0$, see \cite{damm, redmannspa2}. Given an $\epsilon>0$, we conclude that $L(\epsilon X)=-\epsilon Y$. We immediately see that $U(\epsilon X)= \gamma I - \epsilon\sum_{i, j=1}^q M_i^\top X M_j k_{ij}>0$ for a sufficiently small $\epsilon$ exploiting that $\sum_{i, j=1}^q M_i^\top X M_j k_{ij}\geq 0$. Let $\lambda\geq 0$ denote an arbitrary eigenvalue of $\sum_{i, j=1}^q M_i^\top X M_j k_{ij}$. Then, $\frac{1}{\gamma-\epsilon \lambda}$ is an eigenvalue of $U(\epsilon X)^{-1}$. If $\lambda_{\max}$ is the largest out of all these eigenvalues, we obtain\begin{align*}
 \epsilon Y -S(\epsilon X) U(\epsilon X)^{-1} S(\epsilon X)^\top &\geq \epsilon Y -\epsilon ^2 S(X)S(X)^\top   \frac{1}{\gamma-\epsilon \lambda_{\max}}  \\
 &\geq I (\epsilon \lambda^Y_{\min} -\epsilon ^2 \lambda_{\max}^{SS^\top}   \frac{1}{\gamma-\epsilon \lambda_{\max}}) >0                                                                                                                                                                                                                                                                                                                                                                                                                                                                                                                                                                                                                                                                                                                                                                                                                                                                                           \end{align*}
for $\epsilon >0$ small enough, where $\lambda^Y_{\min}>0$ is the smallest and $\lambda_{\max}^{SS^\top}\geq 0$ the largest eigenvalue of $Y$ and $S(X)S(X)^\top$, respectively. Using this estimate yields $L(\epsilon X)< -S(\epsilon X) U(\epsilon X)^{-1} S(\epsilon X)^\top$ for a sufficiently small $\epsilon$. For that reason, $P=(\epsilon X)^{-1}$ is a Gramian according to Definition \ref{def_gramP}.
\end{proof}

 Since $P>0$ is symmetric, we can find a basis of eigenvectors $(p_k)$ providing a representation for the state variable in \eqref{stochstatenew}:
\begin{align}\label{eigen_rep}
x(t)= \sum_{k=1}^n \langle x(t), p_{k}  
\rangle_2 \,p_k.                                                                                                                                                                                                                                                                                                                                                                                                                                                                                                                                                                                                                                                                                                                                                                                                                                                                                                                                                                                                                                                        \end{align}
The next theorem answers the question which directions $p_k$ can be neglected in \eqref{stochstatenew}.
\begin{thm}\label{energy_est}
Let $P$ be as in Definition \ref{def_gramP} and $x$ be the solution of \eqref{stochstatenew}. Moreover, let $(p_k, \lambda_{k})$ be an associated basis of eigenvectors giving us \eqref{eigen_rep} together with the eigenvalues $\lambda_{k}>0$. With $x_0=0$, we have
\begin{align*}
\sup_{t\in[0, T]}\mathbb E \langle x(t), p_{k}  \rangle_2^2 &\leq \lambda_{k} \gamma\left\|u\right\|_{L^2_T}^2,
\end{align*}
where $\gamma>0$ enters the definition of the operator $U$.
\end{thm}
\begin{proof}
We apply Ito's product rule and subsequently take the expectation, see Lemma \ref{lemstochdiff} for details. Hence,  \begin{align*}
 \frac{d}{dt}\mathbb E\left[x(t)^\top P^{-1} x(t)\right]= & 2 \mathbb E\left[x(t)^\top P^{-1} [Ax(t)+Bu(t)]\right]  \\
 &+ \sum_{i, j=1}^q \mathbb E\left[\big(N_i x(t) +M_i u(t)\big)^\top P^{-1} \big(N_j x(t) +M_j u(t)\big)\right]k_{ij}.
 \end{align*}
We integrate this equation over $[0, t]$ with $t\leq T$ yielding
\begin{align*}
 \mathbb E\left[x(t)^\top P^{-1} x(t)\right]&=  \int_0^t \mathbb E\left[x(s)^\top L(P^{-1}) x(s)\right]+2  \mathbb E\left[x(s)^\top P^{-1}B u(s) \right]\\
 &\quad \quad+ \sum_{i, j=1}^q \mathbb E\left[2 x(s)^\top N_i^\top P^{-1} M_j u(s)+u(s)^\top M_i^\top P^{-1} M_j u(s)\right]k_{ij}ds\\
&=  \int_0^t \mathbb E\left[x(s)^\top L(P^{-1}) x(s)\right]+2 \mathbb E\left[x(s)^\top S(P^{-1}) u(s) \right]\\
 &\quad \quad + \gamma \mathbb E\left[u(s)^\top u(s)\right]-\mathbb E\left[u(s)^\top U(P^{-1}) u(s)\right]ds.
 \end{align*}
 We insert \eqref{gramP} into the above equation and obtain \begin{align*}
 \mathbb E\left[x(t)^\top P^{-1} x(t)\right]&\leq \int_0^t -\mathbb E\left[x(s)^\top S(P^{-1}) U(P^{-1})^{-1} S(P^{-1})^\top  x(s)\right]\\
 &\quad \quad +2 \mathbb E\left[x(s)^\top S(P^{-1}) u(s) \right] + \gamma \mathbb E\|u(s)\|_2^2-\mathbb E\left[u(s)^\top U(P^{-1}) u(s)\right]ds \\
 &= \mathbb E \int_0^t -\|U(P^{-1})^{-\frac{1}{2}} S(P^{-1})^\top  x(s)\|_2^2-\|U(P^{-1})^\frac{1}{2} u(s)\|_2^2 ds\\
 &\quad \quad +\mathbb E \int_0^t 2 \langle S(P^{-1})^\top x(s), u(s) \rangle_2 ds + \gamma \|u\|_{L^2_t}^2\\
 &=\gamma \|u\|_{L^2_t}^2- \mathbb E \int_0^t \|U(P^{-1})^{-\frac{1}{2}} S(P^{-1})^\top  x(s)-U(P^{-1})^\frac{1}{2} u(s)\|_2^2 ds\\
 &\leq \gamma \|u\|_{L^2_T}^2
 \end{align*}
exploiting that $U(P^{-1})$ is positive definite. The statement follows by the observation that
 \begin{align*}
 \langle x(t), p_{k}  \rangle_2^2 &\leq \lambda_{k}\; \sum_{i=1}^n \lambda_{i}^{-1} \langle x(t), p_{i}  
\rangle_2^2 =\lambda_{k} \Big\|\sum_{i=1}^n \lambda_{i}^{-\frac{1}{2}} \langle x(t), p_{i}  \rangle_2 \;p_{i}\Big\|_{2}^2
 =\lambda_{k} \Big\|P^{-\frac{1}{2}} x(t)\Big\|_{2}^2\\
 &= \lambda_{k} \; x(t)^\top P^{-1} x(t).
\end{align*}
This concludes the proof.
\end{proof}
 Theorem \ref{energy_est} tells us that all eigenspaces corresponding to small eigenvalues of $P$ can be removed only causing a small error. This will be exploited in the dimension reduction procedure later. However, we need to be able to solve for $P$. Therefore, let us briefly comment on the computation of this Gramian below.
\begin{remark}
 Solving \eqref{gramP} seems to be very tricky and is a bottleneck of our approach as long as no efficient solver is available for such matrix inequalities. In fact, a straight forward approach is to solve for $P^{-1}$ instead of computing $P$ directly. This can be done by
reformulating \eqref{gramP} as a linear matrix inequality based on Schur complement conditions for the definiteness of a matrix. This yields
 \begin{align}\label{LMI_remark}
 \begin{bmatrix}
 {L(P^{-1})} & {S(P^{-1})} \\
 {S(P^{-1})^\top}& {-U(P^{-1})}
 \end{bmatrix}\leq 0, \quad P^{-1}, U(P^{-1})>0,
\end{align}
maximizing, e.g., the trace of $P^{-1}$. Here, the hope is to find a $P^{-1}$, such that $P$ has a large number of small eigenvalues. This is beneficial within the dimension reduction approach due to Theorem \ref{energy_est}. Now, a linear matrix inequality solver (e.g. YALMIP and MOSEK \cite{yalmip, mosek}) can find a solution to the maximization of $\trace(P)$ subject to \eqref{LMI_remark}. 
\end{remark}
Below, we aim to investigate which directions in the state variable $x$ (with zero initial condition) are of low relevance for the quantity of interest $y$ in \eqref{output_eq}. In this context, let us introduce a second Gramian.
\begin{defn}\label{def_gramQ}
Given $x_0=0$, an observability Gramian for \eqref{original_system} is a symmetric matrix $Q\geq 0$ satisfying \begin{align}\label{gramQ}
L(Q)+C^\top C \leq 0.                                                                                                                                                        \end{align}
\end{defn}
Choosing an arbitrary $t_0\in [0, \infty)$, we denote the solution to \eqref{stochstatenew} by $x(t; t_0,  x_0, u)$ for the moment with $t_0$ representing the initial time and $t\geq t_0$. Further, if $t_0=0$, we set $x(t):=x(t; 0, x_0, u)$. Exploiting the linearity of the solution in $x_0$ and $u$, we obtain \begin{align}\label{decompxt}
x(t)= x(t; 0, x_0, u) = x(t; t_0, x(t_0), u)  = x(t; t_0, x(t_0), 0)             +x(t; t_0, 0, u).                                                                                                                                                                                                                                                                                                                                                                                                                                                                                                     \end{align}
Based on \eqref{decompxt}, we aim to answer which directions in $x(t_0)$ are less relevant for the output $y$. Exploiting that an observability Gramian according to Definition \ref{def_gramQ} is symmetric, we choose a basis of eigenvectors $(q_k)$ of $Q$ for $\mathbb R^n$ yielding 
\begin{align}\label{eigen_rep2}
x(t_0)= \sum_{k=1}^n \langle x(t_0), q_{k}  
\rangle_2 \,q_k.                                                                                                                                                                                                                                                                                                                                                                                                                                                                                                                                                                                                                                                                                                                                                                                                                                                                                                                                                                                                                                                        \end{align}
Inserting \eqref{eigen_rep2} into \eqref{decompxt} and using the linearity in the initial state once more, we have 
\begin{align}\label{unwichtigeqk}
 y(t)= \sum_{k=1}^n \langle x(t_0), q_{k}  
\rangle_2  Cx(t; t_0, q_k, 0)             +Cx(t; t_0, 0, u)                                                                                                                                                                                                                                                                                                                                                                                                                                                                                                     \end{align}
for $t\geq t_0$. In order to answer which directions $q_k$ can be neglected in $y$, we formulate the following.
\begin{prop}\label{prop_qk_unimportatn}
Let $(q_k)$ be eigenvectors of $Q$ that represent an orthonormal basis of $\mathbb R^n$ with associated eigenvalues $(\mu_k)$. Let $x(t; t_0, q_k, 0)$ be the solution of \eqref{stochstatenew} with $x_0=q_k$, $u\equiv 0$ and $t_0$ being the initial time. Then, we have \begin{align*}
 \mathbb E\int_{t_0}^\infty  \Big\|\langle x(t_0), q_{k}  
\rangle_2  Cx(t; t_0, q_k, 0) \Big\|_{2}^2  dt \leq \mu_k  \mathbb E \langle x(t_0), q_{k}  
\rangle_2^2.                                                                                                                                                                                                                                                                                                                                                                                                                                                                                                                                    \end{align*}
\end{prop}
\begin{proof}
 We obtain that \begin{align*}
&\mathbb E \Big\|\langle x(t_0), q_{k}  
\rangle_2  Cx(t; t_0, q_k, 0) \Big\|_{2}^2 = \mathbb E \Big[\mathbb E\Big\{\langle x(t_0), q_{k}  
\rangle_2^2\Big\| Cx(t; t_0, q_k, 0) \Big\|_{2}^2 \Big\vert \mathcal F_{t_0}\Big\}\Big]\\
&= \mathbb E \Big[\langle x(t_0), q_{k}  
\rangle_2^2\mathbb E\Big\{\Big\| Cx(t; t_0, q_k, 0) \Big\|_{2}^2 \Big\vert \mathcal F_{t_0}\Big\}\Big]= \mathbb E \Big[\langle x(t_0), q_{k}  
\rangle_2^2\Big]\mathbb E\Big\| Cx(t; t_0, q_k, 0) \Big\|_{2}^2
                \end{align*}
exploiting that $x(t_0)$ is $\mathcal F_{t_0}$-measurable and $x(t; t_0, q_k, 0)$ is independent of $\mathcal F_{t_0}$. Using the arguments in the proof of Theorem \ref{energy_est} based on Ito's product rule (see also Lemma \ref{lemstochdiff}), we obtain
\begin{align*}
 \frac{d}{dt}\mathbb E\left[x(t; t_0, q_k, 0)^\top Q x(t; t_0, q_k, 0)\right]= & 2 \mathbb E\left[x(t; t_0, q_k, 0)^\top Q Ax(t; t_0, q_k, 0)\right]  \\
 &+ \sum_{i, j=1}^q \mathbb E\left[\big(N_i x(t; t_0, q_k, 0) \big)^\top Q N_j x(t; t_0, q_k, 0) \right]k_{ij}.
 \end{align*}
 We integrate this equation over $[t_0, t]$ giving us
\begin{align*}
 \mathbb E\left[x(t; t_0, q_k, 0)^\top Q x(t; t_0, q_k, 0)\right]= q_k^\top Q q_k+  \int_{t_0}^t \mathbb E\left[x(s; t_0, q_k, 0)^\top L(Q) x(s; t_0, q_k, 0)\right]ds.
\end{align*}
We insert \eqref{gramQ} into this equation resulting in \begin{align*}
 \int_{t_0}^t \mathbb E\Big\| Cx(s; t_0, q_k, 0) \Big\|_{2}^2 ds \leq \mu_k                                                                                              \end{align*}
and the result follows by taking $t\rightarrow\infty$.
 \end{proof}
 Proposition \ref{prop_qk_unimportatn} together with \eqref{unwichtigeqk} tell us that directions $q_k$ of $x(t_0)$ in \eqref{eigen_rep2} contribute very little to $y(t)$, $t\geq t_0$ for each $t_0\in[0, \infty)$ in case $\mu_k$ is small. Such directions will be neglected in the reduced dynamics.

 \subsection{Balancing transformation and reduced order model}\label{sec_BT}
Below, we further suppose that $x_0=0$.   
Using a regular matrix
$T\in\mathbb{R}^{n\times n}$, we introduce a balanced state variable by $x_b=Tx$. This transformation is chosen in a particular way in order to simultaneously diagonalize $P$ and $Q$ introduced in Definitions \ref{def_gramP} and \ref{def_gramQ}. Such a $T$ always exists under the condition that $Q>0$. As a consequence of the diagonalization, the associated bases of eigenvectors $(p_k)$ and $(q_k)$  will be the canonical basis of $\mathbb R^n$. Referring to Theorem \ref{energy_est} and Proposition \ref{prop_qk_unimportatn} this means that  components in the transformed state variable $x_b$ turn out to be of low relevance if the associated  diagonal entries of the diagonalized Gramians are small.  We now insert $x=T^{-1}x_b$ into \eqref{original_system}  leading to a stochastic
system with coefficients
\begin{align*}
  (A_b, B_b, M_{i, b}, N_{i, b}, C_b):=(TAT^{-1},TB, TM_i, T N_i T^{-1},CT^{-1})
\end{align*}
instead of the original ones $(A, B, M_i, N_i,C)$, i.e.,  \begin{subequations}\label{original_system_trans}
\begin{align}\label{stochstate_trans}
             dx_b(t)&=[A_b x_b(t)+B_b u(t)]dt+\sum_{i=1}^q [N_{i, b} x_b(t)+M_{i, b} u(t)]dw_i(t),\\ \label{output_eq_trans}
            y(t) &= C_bx_b(t)+D u(t),\quad x_b(0)=0, \quad t\in [0, T].
\end{align}
\end{subequations}
Systems \eqref{original_system} and \eqref{original_system_trans} have the same input-output behaviour and are therefore considered to be equivalent. They share properties such as asymptotic stability. However, the Gramians are different. The particular choice of $T$ as well as the form of the Gramians of \eqref{original_system_trans} are stated in the following proposition.
\begin{prop}\label{prop_bal}
 Suppose that $T$ is an invertible matrix. Given the Gramians $P$ and $Q$ of the original system \eqref{original_system} with $x_0=0$ according to Definitions \ref{def_gramP} and \ref{def_gramQ}. Then, the Gramians of \eqref{original_system_trans} are $P_b= TPT^\top$ and $Q_b=T^{-\top}QT^{-1}$. Suppose that $Q>0$, we can find a simultaneous diagonalization, i.e., $P_b=Q_b = \Sigma= \diag(\sigma_1,\ldots,\sigma_n) $ using the balancing transformation \begin{align}\label{bal_transform}
       T=\Sigma^{\frac{1}{2}} U^\top L_P^{-1},                                                                                                                                                                                                                                                                                                                                                                                                                                                                                                                                                                                                                                                                                                                                                                                          \end{align} 
where $P=L_PL_P^\top$ and $L_P^\top QL_P=U\Sigma^2 U^\top$ is a spectral factorization  with an orthogonal $U$.
\end{prop}
\begin{proof}
We multiply \eqref{gramP} with $T^{-\top}$ from the left and with $T^{-1}$ from the right. This results in \begin{align*}
T^{-\top} L(P^{-1})T^{-1}+T^{-\top} S(P^{-1}) U(P^{-1})^{-1} S(P^{-1})^\top T^{-1} \leq 0.                                                                                                                                                        \end{align*}
We look at each term individually and observe  \begin{align*}
T^{-\top} L(P^{-1})T^{-1}& =T^{-\top}  A^\top T^\top T^{-\top} P^{-1} T^{-1}+ T^{-\top} P^{-1} T^{-1}T A T^{-1}\\
&\quad+\sum_{i, j=1}^q T^{-\top}  N_i^\top T^\top T^{-\top} P^{-1} T^{-1} T N_j T^{-1} k_{ij}\\
&=A_b^\top P_b^{-1} + P_b^{-1} A_b+\sum_{i, j=1}^q N_{i, b}^\top  P_b^{-1} N_{j, b}  k_{ij}
\end{align*}
as well as \begin{align*}
T^{-\top} S(P^{-1}) & = T^{-\top}  P^{-1} T^{-1} T B+ \sum_{i, j=1}^q T^{-\top}  N_i^\top T^\top T^{-\top} P^{-1} T^{-1} T M_j  k_{ij}  \\ & = P_b^{-1} B_b + \sum_{i, j=1}^q N_{i, b}^\top  P_b^{-1} M_{j, b}  k_{ij}
           \end{align*}
and $U(P^{-1}) = \gamma I - \sum_{i, j=1}^q M_i^\top T^\top T^{-\top} P^{-1} T^{-1} T M_j  k_{ij}= \gamma I - \sum_{i, j=1}^q M_{i, b}^\top P_b^{-1} M_{j, b}  k_{ij}$. This shows that $P_b$ is the reachability Gramian of \eqref{original_system_trans}. Let us now multiply \eqref{gramQ} with $T^{-\top}$ from the left and with $T^{-1}$ from the right. Then, \begin{align*}
 T^{-\top}L(Q)T^{-1}+C_b^\top C_b \leq 0                                                                                                                                                                                                                                                                                                                                                                                                                                                                                                                                                                                                                                                                                                                                           \end{align*}
 shows that $Q_b$ is the observability Gramian of \eqref{original_system_trans} taking into account that $T^{-\top}L(Q)T^{-1}=A_b^\top Q_b + Q_b A_b+\sum_{i, j=1}^q N_{i, b}^\top  Q_b N_{j, b}  k_{ij}$. Since $P$ and $Q$ are invertible, the same is true for $L_P$ and $\Sigma$. Therefore, $T$ and its inverse $T^{-1}= L_P U \Sigma^{-\frac{1}{2}}$ are well-defined. Inserting this into the definitions of $P_b$ and $Q_b$ provides that both Gramians are equal to $\Sigma$. This completes the proof.
\end{proof}
Since $L_P^\top QL_P$ has the same eigenvalues as $PQ$ it becomes clear that the Hankel singular values (HSVs) $\sigma_i$ are given by $\sigma_i= \sqrt{\lambda_i(PQ)}$, where $\lambda_i(\cdot)$ denotes the $i$th eigenvalue of a matrix with real spectrum. We partition the balanced state   $x_b=\begin{bmatrix}
x_{1}\\ x_{2}                                         
\end{bmatrix}$, $x_1\in \mathbb R^r$, and the diagonal matrix of HSVs
$\Sigma=\diag(\Sigma_1,\Sigma_{2})$, where
  $\Sigma_1= \diag(\sigma_1,\ldots,\sigma_r)$ contains the large and $\Sigma_{2}=\diag(\sigma_{r+1},\ldots,\sigma_n)$ the small HSVs, $r<n$. Accordingly, we partition the balanced realization as follows:
  \begin{equation}\label{part_bal}
  \begin{aligned}
A_b= \begin{bmatrix}{A}_{11}& {A}_{12}\\ 
 {A}_{21}& {A}_{22}\end{bmatrix},\quad B_b &= \begin{bmatrix}{B}_1\\ B_2\end{bmatrix}, \quad M_{i, b} &= \begin{bmatrix}{M}_{i, 1}\\ M_{i, 2}\end{bmatrix} 
  \end{aligned},\quad N_{i, b}= \begin{bmatrix}{N}_{i, 11}&{N}_{i, 12} \\ 
{N}_{i, 21}& {N}_{i, 22}\end{bmatrix}, \quad C_b= \begin{bmatrix}{C}_1 &
C_2\end{bmatrix}.
  \end{equation}
 We set $x_2\equiv 0$ in the equations of $x_1$ and $y$, since $x_2$ is of low relevance. Truncating the equation for $x_2$ leads to the following reduced order model \begin{subequations}\label{rom_x00}
\begin{align}\label{red_stochstate_x00}
             dx_r(t)&=[A_{11} x_r(t)+B_1 u(t)]dt+\sum_{i=1}^q [N_{i, 11} x_r(t)+M_{i, 1} u(t)]dw_i(t),\\ \label{red_output_eq_x00}
            y_r(t) &= C_1x_r(t)+D u(t),\quad x_r(0)=0, \quad t\in [0, T].
\end{align}
\end{subequations} 
For simplicity of the notation, let us for the remainder of this subsection assume that the original system is already balanced, i.e., $x=x_b$ as well as $(A, B, M_i, N_i, C)=(A_b, B_b, M_{i, b}, N_{i, b}, C_b)$. Consequently, we deal with matrix inequalities of the following form \begin{align}\label{gramPQ_diag}
L(\Sigma^{-1})+S(\Sigma^{-1}) U(\Sigma^{-1})^{-1} S(\Sigma^{-1})^\top \leq 0, \quad \text{and}\quad
L(\Sigma)+C^\top C \leq 0.                                                                                                                                                     \end{align}
We now derive the same type of matrix inequalities for the reduced system \eqref{rom_x00}.
\begin{prop}\label{prop_red_bal}
 Given the balanced realization $(A, B, M_i, N_i, C)$ and introducing the reduced operators\begin{align*}                                                                                                                                                                                                                                                             L_{11}(X)&:= A_{11}^\top X + X A_{11} + \sum_{i, j=1}^q N_{i, 11}^\top X N_{j, 11} k_{ij},\\                                                                                                                                                                                                                                                          S_1(X)&:= X B_1 + \sum_{i, j=1}^q N_{i, 11}^\top X M_{j, 11} k_{ij},\\                                                                                                                                                                                                                                                                U_1(X)&:= \gamma I - \sum_{i, j=1}^q M_{i, 11}^\top X M_{j, 11} k_{ij}, \quad \gamma>0,                                                                                                                                                                                                                                                                \end{align*}
for a generic $X\in \mathbb R^{r\times r}$, it holds that $U_1(\Sigma_1^{-1})>0$ and \begin{align}\label{red_reach} 
L_{11}(\Sigma_1^{-1})+S_1(\Sigma_1^{-1}) U_1(\Sigma_1^{-1})^{-1} S_1(\Sigma_1^{-1})^\top \leq 0                                                                                                                                                      \end{align}
as well as \begin{align}\label{red_obs}
L_{11}(\Sigma_1)\leq C_1^\top C_1.
\end{align}
This means that $\Sigma_1$ is a Gramian of the balanced reduced system \eqref{rom_x00}.
\end{prop}
\begin{proof}
 By the definiteness criterion of Schur, the first inequality in \eqref{gramPQ_diag} is equivalent to \begin{align}\label{schur_rep}
 \begin{bmatrix}-L(\Sigma^{-1})& S(\Sigma^{-1})\\ 
 S(\Sigma^{-1})^\top& U(\Sigma^{-1})\end{bmatrix}  \geq 0.                                                                                                                                                                                                    \end{align}
We multiply this inequality with $\begin{bmatrix} v\\ 
 z\end{bmatrix}^\top$ from the left and with $\begin{bmatrix} v\\ 
 z\end{bmatrix}$ from the right, where $v\in \mathbb R^n$ and $z\in \mathbb R^m$ yielding \begin{align*}
 -v^\top L(\Sigma^{-1}) v + 2v^\top S(\Sigma^{-1}) z+ z^\top U(\Sigma^{-1}) z  \geq 0.                                                                                                                                                                                                    \end{align*}
 We use the partition in \eqref{part_bal} and the definition of the reduced operators. Now, setting $v=\begin{bmatrix} v_1\\ 
 0\end{bmatrix}$, where $v_1\in \mathbb R^r$, we obtain \begin{align*}
 0\leq &-v_1^\top\Big( L_{11}(\Sigma_1^{-1}) + \sum_{i, j=1}^q N_{i, 21}^\top \Sigma_2^{-1} N_{j, 21} k_{ij} \Big)v_1\\
 &+ 2v_1^\top\Big( S_1(\Sigma_1^{-1})+  \sum_{i, j=1}^q N_{i, 21}^\top \Sigma_2^{-1} M_{j, 2} k_{ij} \Big)z
 + z^\top\Big(U_1(\Sigma_1^{-1})- \sum_{i, j=1}^q M_{i, 2}^\top \Sigma_2^{-1} M_{j, 2} k_{ij}\Big) z.                                                                                                                                                                                                    \end{align*}
 It remains to show that $-v_1^\top \sum_{i, j=1}^q N_{i, 21}^\top \Sigma_2^{-1} N_{j, 21} k_{ij} v_1+2v_1^\top \sum_{i, j=1}^q N_{i, 21}^\top \Sigma_2^{-1} M_{j, 2} k_{ij} z-z^\top \sum_{i, j=1}^q M_{i, 2}^\top \Sigma_2^{-1} M_{j, 2} k_{ij}z\leq 0$, since 
 \begin{align*}
-v_1^\top L_{11}(\Sigma_1^{-1}) v_1+2v_1^\top S_1(\Sigma_1^{-1})z +z^\top U_1(\Sigma_1^{-1})z \geq 0                                                                                                                                                                                                   \end{align*}
for all $v_1\in \mathbb R^r$ and $z\in\mathbb R^m$ is equivalent to \eqref{red_reach}. In order to show the missing step for \eqref{red_reach}, we observe that \begin{align*}
 \sum_{i, j=1}^q N_{i, 21}^\top \Sigma_2^{-1} N_{j, 21} k_{ij}&=N_{21}^\top \big(K\otimes \Sigma_2^{-1}\big) N_{21},\\
  \sum_{i, j=1}^q N_{i, 21}^\top \Sigma_2^{-1} M_{j, 2} k_{ij}&=N_{21}^\top \big(K\otimes \Sigma_2^{-1}\big) M_{2},\\
  \sum_{i, j=1}^q M_{i, 2}^\top \Sigma_2^{-1} M_{j, 2} k_{ij}&=M_{2}^\top \big(K\otimes \Sigma_2^{-1}\big) M_{2},                                                                                                                                      \end{align*}
where $N_{21}^\top:=\begin{bmatrix}
 N_{1, 21}^\top & \dots & N_{q, 21}^\top                     
                    \end{bmatrix}$ and $M_{2}^\top:=\begin{bmatrix}
 M_{1, 2}^\top & \dots & M_{q, 2}^\top                     
                    \end{bmatrix}$. Since $K=(k_{ij})$ and $\Sigma_2$ are 
positive (semi-)definite, we have that $K\otimes \Sigma_2^{-1}$ is positive (semi-)definite. Therefore, we obtain \begin{align*}
 &-v_1^\top N_{21}^\top \big(K\otimes \Sigma_2^{-1}\big) N_{21} v_1+2v_1^\top N_{21}^\top \big(K\otimes \Sigma_2^{-1}\big) M_{2} z-z^\top M_{2}^\top \big(K\otimes \Sigma_2^{-1}\big) M_{2} z\\
 &=-\Big\| \big(K\otimes \Sigma_2^{-1}\big)^{\frac{1}{2}} N_{21} v_1-\big(K\otimes \Sigma_2^{-1}\big)^{\frac{1}{2}} M_{2} z\Big\|_{2}^2 \leq 0                                                                                                                                                                                                   \end{align*}
finally providing \eqref{red_reach}. We show \eqref{red_obs} using partition \eqref{part_bal} when evaluating the left upper block of the second inequality in \eqref{gramPQ_diag}. This yields \begin{align*}
  L_{11}(\Sigma_1) + \sum_{i, j=1}^q N_{i, 21}^\top \Sigma_2 N_{j, 21} k_{ij}    \leq -C_1^\top C_1.                                                                                                                                                                                                                                                                                                                                                                                          \end{align*}
Since $\sum_{i, j=1}^q N_{i, 21}^\top \Sigma_2 N_{j, 21} k_{ij} =N_{21}^\top \big(K\otimes \Sigma_2\big) N_{21} \geq 0$, the proof is complete.
\end{proof}
We are ready to prove a bound for the error between the outputs of \eqref{original_system} and \eqref{rom_x00}.
\begin{thm}\label{thm_error_bound}
Let $y$ be the output of \eqref{original_system} with $x(0)=0$, for which we assume to have a balanced realization, e.g., the assumptions of Proposition \ref{prop_bal} hold true. Given the $r$-dimensional reduced system \eqref{rom_x00} with output $y_r$  and $x_r(0)=0$. Then,  for all $u\in L^2_T$, we have \begin{align*}
\left\| y-y_{r}\right\|_{L^2_T}\leq 2\gamma^\frac{1}{2}\sum_{k=r+1}^n\sigma_k \|u\|_{L^2_T}.                                                                                                                                                                                                                                                                                                                                                                                                                                                                                                                                                                                    \end{align*}
\end{thm}
\begin{proof}
Below, we continue making use of the partition in \eqref{part_bal} while assuming that our original system is already balanced. We add a zero line to the reduced system \eqref{rom_x00} yielding                                                                                                                                                                                                                                                                                                                                                                                                                                                                                                                     \begin{align}\label{bal_par_minus1} 
 d\smat x_r(t)\\0\srix = &\big[A \smat x_r(t)\\0\srix  +  B u(t)-\smat 0 \\ v_0(t)\srix\big]dt
 + \sum_{i=1}^{q}\big[ N_{i} \smat x_r(t)\\0\srix+ M_i u(t) -\smat 0 \\ v_i(t)\srix\big] dw_i(t),
\end{align}
where $v_0= A_{21}x_r+B_2 u$ and $v_i=N_{i, 21} x_r + M_{i, 2} u$. We introduce $x_-(t):= x(t) -\smat x_{r}(t)\\0\srix $ and $x_+(t):= x(t) +\smat x_{r}(t)\\0\srix $. Combining \eqref{original_system} with \eqref{bal_par_minus1} gives us \begin{align*}
 dx_-(t) &= [A x_-(t)+\smat 0 \\ v_0(t)\srix]dt + \sum_{i=1}^{q}\big[ N_{i} x_-(t) + \smat 0 \\ v_i(t)\srix\big]dw_i(t),\\ 
  dx_+(t) &= [A x_+(t)+2Bu(t)-\smat 0 \\ v_0(t)\srix]dt + \sum_{i=1}^{q}\big[ N_{i} x_+(t)+2M_iu(t) - \smat 0 \\ v_i(t)\srix\big]dw_i(t).
\end{align*} 
We begin with the case of $\Sigma_2=\sigma I$, $\sigma>0$, which is the foundation for showing the error bound in full generality.
In order to proof the bound, we need to derive equations for quadratic forms of $x_-$ and $x_+$. First, Ito's product rule, see Lemma \ref{lemstochdiff}, yields \begin{align*}
 \mathbb E\left[x_-(t)^\top \Sigma x_-(t)\right]&= 2\int_0^t \mathbb E\left[x_-(s)^\top \Sigma[A x_-(s)+\smat 0 \\ v_0(s)\srix\right]ds \\
 &\quad+ \int_0^t\sum_{i, j=1}^q \mathbb E\left[\big( N_{i} x_-(s) + \smat 0 \\ v_i(s)\srix\big)^\top \Sigma \big( N_{j} x_-(s) + \smat 0 \\ v_j(s)\srix\big)\right]k_{ij} ds\\
 &=\int_0^t \mathbb E\left[x_-(s)^\top L(\Sigma) x_-(s)\right]ds+2\int_0^t \mathbb E\left[x_-(s)^\top \Sigma\smat 0 \\ v_0(s)\srix\right] ds\\
 &\quad +\int_0^t\sum_{i, j=1}^q \mathbb E\left[\big( 2 N_{i} x_-(s) + \smat 0 \\ v_i(s)\srix\big)^\top \Sigma  \smat 0 \\ v_j(s)\srix\right]k_{ij} ds.
 \end{align*}  
 We see that $x_-(s)^\top \Sigma\smat 0 \\ v_0(s)\srix= x_2(s)^\top \Sigma_2 v_0(s)$ and $\big( 2 N_{i} x_-(s) + \smat 0 \\ v_i(s)\srix\big)^\top \Sigma  \smat 0 \\ v_j(s)\srix=(2\smat N_{i, 21} & N_{i, 22}\srix (x(s) -\smat x_{r}(s)\\0\srix )+v_i(s))^\top \Sigma_2 v_j(s)=(2\smat N_{i, 21} & N_{i, 22}\srix x(s) - v_i(s)+2 M_{i, 2} u(s))^\top \Sigma_2 v_j(s)$. Consequently, inserting \eqref{gramQ} into the above identify for $\mathbb E\left[x_-(t)^\top \Sigma x_-(t)\right]$ leads to \begin{align*}
 \mathbb E\left[x_-(t)^\top \Sigma x_-(t)\right]&\leq \int_0^t -\mathbb E\left\| C x_-(s)\right\|_2^2 ds+2\int_0^t \mathbb E\left[x_2(s)^\top \Sigma_2 v_0(s)\right] ds\\
 &\quad +\int_0^t\sum_{i, j=1}^q \mathbb E\left[(2\smat N_{i, 21} & N_{i, 22}\srix x(s) - v_i(s)+2 M_{i, 2} u(s))^\top \Sigma_2 v_j(s)\right]k_{ij} ds.
 \end{align*}  
Since $C x_-= y-y_r$, we obtain that \begin{align*}
\left\|y-y_r\right\|_{L^2_t}^2  &\leq 2\int_0^t \mathbb E\left[x_2(s)^\top \Sigma_2 v_0(s)\right] ds\\
 &\quad +\int_0^t\sum_{i, j=1}^q \mathbb E\left[(2\smat N_{i, 21} & N_{i, 22}\srix x(s) - v_i(s)+2 M_{i, 2} u(s))^\top \Sigma_2 v_j(s)\right]k_{ij} ds.
 \end{align*}  
We apply the assumption $\Sigma_2=\sigma I$ and exploit that $\sum_{i, j=1}^q v_i(s)^\top v_j(s)k_{ij}\geq 0$. Hence, 
\begin{align}\label{est_first_step}
\left\|y-y_r\right\|_{L^2_t}^2  &\leq 2\sigma \int_0^t \mathbb E\left[x_2(s)^\top v_0(s)\right] ds\\ \nonumber
 &\quad +\sigma \int_0^t\sum_{i, j=1}^q \mathbb E\left[(2\smat N_{i, 21} & N_{i, 22}\srix x(s) + v_i(s)+2 M_{i, 2} u(s))^\top  v_j(s)\right]k_{ij} ds.
 \end{align}  
We aim to combine inequality \eqref{est_first_step} with a second one that we obtain from the next step. This is, finding an estimate for a quadratic form of $x_+$. Ito's product rule (compare with Lemma \ref{lemstochdiff}) leads to {\allowdisplaybreaks \begin{align*}
 &\mathbb E\left[x_+(t)^\top \Sigma^{-1} x_+(t)\right]= 2\int_0^t \mathbb E\left[x_+(s)^\top \Sigma^{-1}[A x_+(s)+2 B u(s)-\smat 0 \\ v_0(s)\srix\right]ds \\
 &+ \int_0^t\sum_{i, j=1}^q \mathbb E\left[\big( N_{i} x_+(s)+2 M_i u(s) - \smat 0 \\ v_i(s)\srix\big)^\top \Sigma^{-1} \big( N_{j} x_-(s) +2M_i u(s) - \smat 0 \\ v_j(s)\srix\big)\right]k_{ij} ds\\
 &=\int_0^t \mathbb E\left[x_+(s)^\top L(\Sigma^{-1}) x_+(s)\right]ds\\
 &\quad -2\int_0^t \mathbb E\left[x_+(s)^\top \Sigma^{-1}\smat 0 \\ v_0(s)\srix\right] ds+4\int_0^t \mathbb E\left[x_+(s)^\top \Sigma^{-1}Bu(s)\right] ds\\
 &\quad +\int_0^t\sum_{i, j=1}^q \mathbb E\left[\big( 2 N_{i} x_+(s)+ 2 M_i u(s) - \smat 0 \\ v_i(s)\srix\big)^\top \Sigma^{-1}  (2 M_i u(s) - \smat 0 \\ v_i(s)\srix)\right]k_{ij} ds\\
 &=\int_0^t \mathbb E\left[x_+(s)^\top L(\Sigma^{-1}) x_+(s)\right]ds+4\int_0^t \mathbb E\left[u(s)^\top  \sum_{i, j=1}^q M_i^\top\Sigma^{-1}M_j k_{ij} u(s)\right] ds\\
 &\quad +4\int_0^t \mathbb E\left[x_+(s)^\top  \sum_{i, j=1}^q N_i^\top\Sigma^{-1}M_j k_{ij} u(s)\right] ds +4\int_0^t \mathbb E\left[x_+(s)^\top \Sigma^{-1}Bu(s)\right] ds\\
 &\quad -\int_0^t\sum_{i, j=1}^q \mathbb E\left[\big( 2 N_{i} x_+(s)+ 4 M_i u(s) - \smat 0 \\ v_i(s)\srix\big)^\top \Sigma^{-1}   \smat 0 \\ v_i(s)\srix\right]k_{ij} ds\\
 &\quad-2\int_0^t \mathbb E\left[x_+(s)^\top \Sigma^{-1}\smat 0 \\ v_0(s)\srix\right] ds.
 \end{align*}  }
 We observe that $x_+(s)^\top \Sigma^{-1}\smat 0 \\ v_0(s)\srix=x_2(s)^\top \Sigma_2^{-1}v_0(s)$ and $(2 N_{i} x_+(s)+ 4 M_i u(s) - \smat 0 \\ v_i(s)\srix\big)^\top \Sigma^{-1}   \smat 0 \\ v_i(s)\srix=(2\smat N_{i, 21} & N_{i, 22}\srix (x(s) +\smat x_{r}(s)\\0\srix )+4M_{i, 2} u(s)-v_i(s))^\top \Sigma_2^{-1} v_j(s)=(2\smat N_{i, 21} & N_{i, 22}\srix x(s) + v_i(s)+2 M_{i, 2} u(s))^\top \Sigma_2^{-1} v_j(s)$. Using the definitions of the operators $U$ and $S$, we therefore have \begin{align*}
 \mathbb E\left[x_+(t)^\top \Sigma^{-1} x_+(t)\right]&=\mathbb E\int_0^t x_+(s)^\top L(\Sigma^{-1}) x_+(s)-4 u(s)^\top U(\Sigma^{-1})u(s)+4x_+(s)^\top S(\Sigma^{-1})u(s) ds   \\
 &\quad -\int_0^t\sum_{i, j=1}^q \mathbb E\left[(2\smat N_{i, 21} & N_{i, 22}\srix x(s) + v_i(s)+2 M_{i, 2} u(s))^\top \Sigma_2^{-1} v_j(s)\right]k_{ij} ds\\
 &\quad-2\int_0^t \mathbb E\left[x_2(s)^\top \Sigma_2^{-1}v_0(s)\right] ds+4 \gamma \int_0^t  \mathbb E\|u(s)\|_2^2 ds.                                                                                                                                                                                                                                                                                                                                                                                                                                                                                                                                                                                                                                                                                                                                                              \end{align*}
We multiply \eqref{schur_rep} with $\begin{bmatrix} x_+(s)\\ 
 -2u(s)\end{bmatrix}^\top$ from the left and with $\begin{bmatrix} x_+(s)\\ 
 -2 u(s)\end{bmatrix}$ from the right in order to obtain $ -x_+(s)^\top L(\Sigma^{-1}) x_+(s) -4 x_+(s)^\top S(\Sigma^{-1}) u(s)+ 4 u(s)^\top U(\Sigma^{-1}) u(s)  \geq 0$. This yields \begin{align*}
 4 \gamma  \|u\|_{L^2_t}^2 &\geq \sigma^{-1}\int_0^t\sum_{i, j=1}^q \mathbb E\left[(2\smat N_{i, 21} & N_{i, 22}\srix x(s) + v_i(s)+2 M_{i, 2} u(s))^\top  v_j(s)\right]k_{ij} ds\\
 &\quad+2\sigma^{-1}\int_0^t \mathbb E\left[x_2(s)^\top v_0(s)\right] ds                                                                                                                                                                                                                                                                                                                                                                                                                                                                                                                                                                                                                                                                                                                                                             \end{align*}
 exploiting that $\Sigma_2=\sigma I$. Combining this estimate with \eqref{est_first_step}, we see that \begin{align}\label{estim_sigID}
 \left\|y-y_r\right\|_{L^2_t}  &\leq 2\sigma \gamma^\frac{1}{2}  \|u\|_{L^2_t}.                                                                                                                                                                                                         \end{align}
Let us turn our attention to the general case. We can bound the $L^2_T$-error of the dimension reduction procedure as follows \begin{align*}
  \left\|y-y_{r}\right\|_{L^2_T} 
 \leq \left\|y-y_{n-1}\right\|_{L^2_T} + \sum_{i=r+1}^{n-1} \left\| y_{k}-y_{k-1}\right\|_{L^2_T},
 \end{align*}
 where $y_{k}$ is the output of the reduced model with dimension $k\in\{r, \dots, n-1\}$. By \eqref{estim_sigID} with $t=T$, we obtain $\left\|y-y_{n-1}\right\|_{L^2_T}\leq 2\sigma_n \gamma^\frac{1}{2}  \|u\|_{L^2_T}$. In the next reduction step from $y_{n-1}$ to $y_{n-2}$ we exploit Proposition \ref{prop_red_bal}. It tells us that $y_{n-1}$ is the output of a balanced system with Gramian $\diag(\sigma_1, \dots, \sigma_{n-1})$. Therefore, we can apply \eqref{estim_sigID} once more. We continue this procedure until we arrive at the reduced system of order $r$ and find in every step that $\left\| y_{k}-y_{k-1}\right\|_{L^2_T}\leq 2\sigma_k \gamma^\frac{1}{2}  \|u\|_{L^2_T}$ concluding this proof.
\end{proof}
Theorem \ref{thm_error_bound} is one of our main results and will be the basis for the error estimates that we provide for the approaches presented in Subsections \ref{approach1} and \ref{approach2}. It contains an a-priori error bound that links the truncated HSVs to the error of the dimension reduction. Let us now turn our attention to a model reduction technique that is needed for the ansatz of Subsection \ref{approach2}.

\section{Model reduction for \eqref{original_system} with $u\equiv 0$}\label{sec4}
We proceed with the uncontrolled case as the second ansatz of Section \ref{approach2} requires a dimension reduction approach for \eqref{aux_var} with $\tilde A=A$ and $\tilde N_i=N_i$. We begin with the definition of a reachability Gramian having a very different structure than the one in Definition \ref{def_gramP}.
\begin{defn}\label{def_gramP2}
Given that $u\equiv 0$, a reachability Gramian for \eqref{original_system} is a symmetric matrix $\mathbf P\geq 0$ satisfying \begin{align}\label{gramP2}
L^*(\mathbf P)+X_0 X_0^\top = 0,                                                                                                                                                       \end{align}
where $L^*$ is the adjoint operator of $L$ given by $L^*(X):= A X + X A^\top + \sum_{i, j=1}^q N_i X N_j^\top k_{ij}$ for some generic matrix $X$ of suitable dimension.
\end{defn}
Using the terminology created in \cite{bennerdammcruz}, $\mathbf P$ has the form of a type I Gramian. This is, e.g., used in \cite{bennerdamm, redmannbenner} in a control system context, whereas here we have $u\equiv 0$ and $x_0\neq 0$. Let us show why it is relevant in out setting.
\begin{prop}\label{prop_dom_x0}
Let $\mathbf P$ be as in Definition \ref{def_gramP2} and $x$ be the mean square asymptotically stable solution of \eqref{stochstatenew} with $u\equiv 0$. Moreover, let $(\mathbf p_{ k}, \mathbf \lambda_{k})$ be an associated basis of eigenvectors for $\mathbb R^n$ with the eigenvalues $\mathbf {\lambda}_{k}\geq 0$. Then,
\begin{align*}
\int_0^\infty\mathbb E \left\vert\langle x(t), \mathbf p_{k} \rangle_2\right\vert^2 dt\leq \mathbf \lambda_{k} \left\|v \right\|_2^2.
\end{align*}
\end{prop}
\begin{proof}
Let $\Phi$ be the fundamental solution of \eqref{stochstatenew}. We obtain that
 \begin{align*}
\int_0^\infty\mathbb E \left\vert\langle x(t), \mathbf p_{k} \rangle_2\right\vert^2 dt &=\int_0^\infty\mathbb E 
\left\vert \langle \Phi(t) X_0 v, \mathbf p_{k}\rangle_2 \right\vert^2 dt=\int_0^\infty \mathbb E \left\vert\langle v, X_0^\top\Phi(t)^\top \mathbf p_{k}\rangle_2 \right\vert^2 dt\\ 
&\leq 
\mathbf p_{k}^\top \int_0^\infty\mathbb E \left[\Phi(t) X_0 X_0^\top \Phi(t)^\top\right]dt\, \mathbf p_{k} \left\|v \right\|_2^2. 
\end{align*}
It is a well-known fact that $\mathbf P:= \int_0^\infty\mathbb E \left[\Phi(t) X_0 X_0^\top \Phi(t)^\top\right]dt$ solves \eqref{gramP2}, see, e.g.,  \cite{damm}. Therefore, we find that $\mathbf p_{k}^\top \int_0^\infty\mathbb E \left[\Phi(t) X_0 X_0^\top \Phi(t)^\top\right]dt\, \mathbf p_{k} =\mathbf \lambda_{k}$ which concludes the proof.
\end{proof}
\begin{remark}
Let $\bar{\mathbf P}$ be a solution of the inequality in \eqref{gramP2}. Then, we have
$L^*(\bar{\mathbf P}-\mathbf P)=-Y\leq 0$ for some $Y\geq 0$. Using the stochastic representation of this matrix equation, we see that $\bar{\mathbf P}-\mathbf P\geq 0$ allowing to replace $\mathbf P$ by $\bar{\mathbf P}$ in Proposition \ref{prop_dom_x0}.
\end{remark}
When $u\equiv 0$, there is no change in the arguments used below Definition \ref{def_gramQ}. Therefore, the observability Gramians are the same regardless of having an uncontrolled setting or not. However, for simplicity we assume that we do not use a more general inequality but rather the solution to the equation. For that reason, we consider the following definition.
\begin{defn}\label{def_gramQ2}
Given that $u\equiv 0$, an observability Gramian of system \eqref{original_system} is a symmetric matrix $\mathbf Q\geq 0$ satisfying \begin{align}\label{gramQ2}
L(\mathbf Q)+C^\top C = 0.                                                                                                                                                        \end{align}
\end{defn}
We aim to diagonalize $\mathbf P$ and $\mathbf Q$ simultaneously for the same reason explained in Section \ref{sec_BT}. The associated state space transformation is here denoted by $\mathbf T$ defining the balanced state variable $\mathbf x_b=\mathbf T x$   satisfying
 \begin{subequations}\label{original_system_trans_x0}
\begin{align}\label{stochstate_trans_x0}
             d\mathbf x_b(t)&=\mathbf A_b \mathbf x_b(t)dt+\sum_{i=1}^q \mathbf N_{i, b} \mathbf x_b(t) dw_i(t),\\ \label{output_eq_trans_x0}
            y(t) &= \mathbf C_b \mathbf  x_b(t),\quad \mathbf x_b(0)=\mathbf  X_{0, b} v, \quad t\in [0, T],
\end{align}
\end{subequations}
with the same output as \eqref{original_system} ($u\equiv 0$) and coefficients
\begin{align*}
  (\mathbf A_b, \mathbf N_{i, b}, \mathbf C_b, \mathbf X_{0, b}):=(\mathbf T A \mathbf T^{-1}, \mathbf T N_i \mathbf T^{-1},C\mathbf T^{-1}, \mathbf T X_0).
\end{align*}
Given $\mathbf P, \mathbf Q>0$, the balancing transformation $\mathbf T$ exists and has exactly the same structure like $T$ in Proposition \ref{prop_bal}. Just the underlying Gramians are different. The structure is identical because the Gramians of \eqref{original_system_trans_x0} are of the form $\mathbf P_b= \mathbf T\mathbf P\mathbf T^\top$ and $\mathbf Q_b=\mathbf T^{-\top}\mathbf Q\mathbf T^{-1}$ again. The only difference in the proof is that the result is here obtained by multiplying  \eqref{gramP2} by $\mathbf T$ from the left and with $\mathbf T^\top$ from the right. The same type of partition \begin{equation*}
  \mathbf A_b= \begin{bmatrix}{\mathbf A}_{11}& {\mathbf A}_{12}\\ 
 {\mathbf A}_{21}& {\mathbf A}_{22}\end{bmatrix},\quad \mathbf X_{0, b} = \begin{bmatrix}{\mathbf X}_{0, 1}\\ {\mathbf X}_{0, 2}\end{bmatrix},\quad \mathbf N_{i, b}= \begin{bmatrix}{\mathbf N}_{i, 11}&{\mathbf N}_{i, 12} \\ 
{\mathbf N}_{i, 21}& {\mathbf N}_{i, 22}\end{bmatrix}, \quad \mathbf C_b= \begin{bmatrix}{\mathbf C}_1 &
\mathbf C_2\end{bmatrix}
  \end{equation*}
like in \eqref{part_bal} leads to the reduced system \begin{subequations}\label{rom_u0}
\begin{align}\label{red_stochstate_u0}
             d\mathbf x_{\mathbf r}(t)&=\mathbf A_{11} \mathbf x_{\mathbf r}(t)dt+\sum_{i=1}^q \mathbf N_{i, 11} \mathbf x_{\mathbf r}(t)dw_i(t),\\ \label{red_output_eq_u0}
            \mathbf y_{\mathbf r}(t) &= \mathbf C_1 \mathbf x_{\mathbf r}(t),\quad \mathbf x_{\mathbf r}(0)=\mathbf X_{0, 1}v, \quad t\in [0, T].
\end{align}
\end{subequations} 
Here, $\mathbf x_{\mathbf r}(t)\in\mathbb R^{\mathbf r}$, where $\mathbf r$ is the number of small HSVs $\theta_i= \sqrt{\lambda_i(\mathbf P \mathbf Q)}$. The dimension of the reduced matrices in \eqref{rom_u0} is accordingly.
\begin{remark}\label{remark_ansatz2}
It is interesting to notice that the model reduction procedure to lower the dimension of the initial state dynamics is less complex than the one for the dynamics of the control. This is due to the choice of Gramian in each scenario. If $x_0=0$ but the control is present, the algorithm of Section \ref{Sec3} requires the computation of a solution $P>0$ satisfying a linear matrix inequality \eqref{gramP}. The scheme in this section, where $u\equiv 0$ and $x_0\neq 0$, uses solutions $\mathbf P$ and $\mathbf Q$ of matrix equations \eqref{gramP2} and \eqref{gramQ2}. Therefore, it is much easier to determine $\mathbf P$ and $\mathbf Q$.
We refer to \cite{damm_matrix_eq, Sim16a} for potential strategies to solve such matrix equations.\end{remark}
It is not hard to find a first error estimate between the output $y$ of \eqref{original_system} with $u\equiv 0$ and $\mathbf y_{\mathbf r}$. Using the fundamental matrix $\mathbf \Phi_{\mathbf r} $ of \eqref{rom_u0}, this is
\begin{align*}
\|y-\mathbf y_{\mathbf r}\|_{L^2_T}^2&=\int_0^T\mathbb E \left\|y(t)-\mathbf y_{\mathbf r}(t) \right\|_2^2 dt=\int_0^T\mathbb E \left\|C\Phi(t)X_0 v-\mathbf C_1 \mathbf \Phi_{\mathbf r} (t) \mathbf X_{0, 1} v \right\|_2^2 dt\\
&\leq \int_0^\infty\mathbb E \left\|C\Phi(t)X_0 v-\mathbf C_1 \mathbf \Phi_{\mathbf r} (t) \mathbf X_{0, 1} \right\|_F^2 dt \left\| v \right\|_2^2. 
\end{align*}
It remains to analyze the integral $\int_0^\infty\mathbb E \left\|C\Phi(t)X_0 v-\mathbf C_1 \mathbf \Phi_{\mathbf r} (t) \mathbf X_{0, 1} \right\|_F^2 dt$ which is finite due to the mean square asymptotic stability of \eqref{original_system}. This was done in \cite[Section 4.2]{redmannbenner}, where the integral has been expressed as a term depending on the matrix $\Theta_2= \diag(\theta_{\mathbf r +1}, \dots, \theta_n)$ of truncated HSVs. We exploit this result in the following theorem.
\begin{thm}\label{thm_error_bound2}
Let $y$ be the output of \eqref{original_system} with $u\equiv 0$ and given that the Gramians $\mathbf P$ and $\mathbf Q$ associated to this uncontrolled case are positive definite. Suppose that $\mathbf 
y_{\mathbf r}$ is the output of the $\mathbf r$-dimensional reduced system \eqref{rom_u0} with $\mathbf x_{\mathbf r}(0)=0$. Then,  for all $u\in L^2_T$, we have
\begin{align*}
 \|y-\mathbf y_{\mathbf r}\|_{L^2_T}^2\leq  \trace(\Theta_2 \mathcal W) \left\| v \right\|_2^2, 
 \end{align*}
where $\mathcal W =  \mathbf B_2 \mathbf B_2^\top+2  \hat{\mathbf P}_{2}\mathbf A_{21}^\top +\sum_{i, j=1}^q (2 \smat {\mathbf N}_{i, 21} & {\mathbf N}_{i, 22} \srix \hat{\mathbf P} -\mathbf N_{i, 21} \mathbf{P_r} )  \mathbf N_{j, 21}^\top k_{ij}$. The matrices 
   $\hat{\mathbf P}=\smat {\hat{\mathbf P}_{1}}\\ {\hat{\mathbf P}_{2}}\srix$ and  $\mathbf{P_r}$ are the solutions of\begin{align*}
\mathbf A_{11}\mathbf{P_r} + \mathbf{P_r} \mathbf A_{11}^\top + \sum_{i, j=1}^q \mathbf N_{i, 11} \mathbf{P_r}  \mathbf N_{j, 11}^\top k_{ij} &= - \mathbf X_{0, 1} \mathbf X_{0, 1}^\top,\\
    \mathbf A_b \hat{\mathbf P} + \hat{\mathbf P} \mathbf A_{11}^\top + \sum_{i, j=1}^q \mathbf N_{i, b} \hat{\mathbf P} \mathbf N_{j, 11}^\top k_{ij} &= - \mathbf X_{0, b} \mathbf X_{0, 1}^\top.
                  \end{align*}
\end{thm}
\begin{proof}
We refer to Section 4.2 and in particular to Proposition 4.6 in \cite{redmannbenner} for a proof of this statement, where uncorrelated L\'evy noise was considered. The proof is identical for correlated Wiener noise.
\end{proof}
Theorem \ref{thm_error_bound2} contains a bound that involves truncated HSVs and hence links these values to the approximation error. However, in contrast to Theorem \ref{thm_error_bound}, it is not computable a-priori. We are now ready to discuss our two techniques used for reducing the dimension of \eqref{original_system}.

\section{Model reduction for general systems \eqref{original_system}}\label{sec5}

We conclude this paper with formulating error bounds for the model reduction strategies proposed in Subsections \ref{approach1} and \ref{approach2}. Here, we make use of the results established in Sections \ref{Sec3} and \ref{sec4}.

\subsection{The approach according to Subsection \ref{approach1}}

The dimension reduction procedure for \eqref{original_system} can be conducted using the approach of Section \ref{Sec3}, which first of all covers the special case of $x_0=0$. However, we can transform \eqref{original_system} into a system with a zero initial state, that is given in  \eqref{structured_changed_system}. The price that needs to be paid is the enlargement of the control space. In fact, the information concerning the initial state is moved to the new control $\tilde u(t)= \smat u(t)\\ u_0(t)v\srix$. Moreover, all control matrices are updated, meaning that we replace $B, M_i$ and $D$ by $\tilde B, \tilde M_i$ and $\tilde D$. We can now apply the model reduction technique of Section \ref{Sec3} to \eqref{structured_changed_system}. It is based on Gramians $\tilde P, Q>0$ satisfying $\tilde U(\tilde P^{-1})>0$ and \begin{align}\label{Gramians_strat1}
L(\tilde P^{-1})+\tilde S(\tilde P^{-1}) \tilde U(\tilde P^{-1})^{-1} \tilde S(\tilde P^{-1})^\top \leq 0 \quad\text{and}\quad   
L(Q)+C^\top C \leq 0,                                                                                                                                                                                                                                                                                                              \end{align}
where the updated operators are defined by
\begin{align*}
\tilde S(X) = X \tilde B + \sum_{i, j=1}^q N_i^\top X \tilde M_j k_{ij} \quad\text{and}\quad                                                                                                                                                                                                                                                                 \tilde U(X) = \tilde \gamma I - \sum_{i, j=1}^q \tilde M_i^\top X \tilde M_j k_{ij}\quad(\tilde \gamma >0).
\end{align*}
We determine the balancing transformation $\tilde T$ according to \eqref{bal_transform} leading to balanced coefficients that we partition as follows 
  \begin{align*}
\tilde T A \tilde T^{-1}&= \begin{bmatrix}{\tilde A}_{11}& \star\\ 
 \star& \star\end{bmatrix},\quad \tilde T \tilde B = \begin{bmatrix}{\tilde B}_1\\ \star\end{bmatrix}, \quad \tilde T \tilde M_{i} = \begin{bmatrix}{\tilde M}_{i, 1}\\ \star\end{bmatrix} 
 ,\\
 \tilde T N_{i} \tilde T^{-1}&= \begin{bmatrix}{\tilde N}_{i, 11}&\star \\ 
\star& \star\end{bmatrix}, \quad C \tilde T^{-1}= \begin{bmatrix}{\tilde C}_1 &
\star\end{bmatrix}.
 \end{align*}
 Consequently, the reduced system is of the form \begin{align}\label{stochstatenew_z_red2}
 \begin{split}
             d\tilde x_r(t)&=\left[\tilde A_{11} \tilde x_r(t)+\tilde B_1 \tilde u(t) \right]dt+\sum_{i=1}^q \left[\tilde N_{i, 11} \tilde x_r(t)+\tilde M_{i, 1}  \tilde u(t)\right]dw_i(t),\\ 
            \tilde y_r(t) &= \tilde C_1 \tilde x_r(t)+ \tilde D \tilde u(t), \quad \tilde x_r(0)=0.
            \end{split}
\end{align}
By Theorem \ref{thm_error_bound}, we know that the error between the original output $y$ in \eqref{output_eq} and $\tilde y_r$ can be bound as follows \begin{align*}
  \left\| y-\tilde y_{r}\right\|_{L^2_T}\leq 2\tilde \gamma^\frac{1}{2}\sum_{k=r+1}^n\tilde{ \sigma}_k \|\tilde u\|_{L^2_T},                                                                                                                                                                                                                                                                             \end{align*}
where $\tilde{\sigma}_k= \sqrt{\lambda_k(\tilde PQ)}$ are the HSVs of the transformed system \eqref{structured_changed_system}. Exploiting that $\|\tilde u\|_{L^2_T}^2=\|u\|_{L^2_T}^2+\| u_0 v\|_{L^2_T}^2$, we obtain the final error estimate that we formulate in a theorem.
\begin{thm}\label{error_bound_approach1}
Let $y$ be the  the quantity of interest in \eqref{original_system} and given that the Gramian $Q$ is positive definite. Let $\tilde y_{r}$ be the output of the reduced system \eqref{stochstatenew_z_red2}. Then, we have
 \begin{align*}
  \left\| y-\tilde y_{r}\right\|_{L^2_T}\leq 2\tilde \gamma^\frac{1}{2}\sum_{k=r+1}^n\tilde{ \sigma}_k \sqrt{\mathbb E \int_0^T \|u(t)\|_2^2 dt+\mathbb E \int_0^T \|u_0(t)\|_2^2 dt \|v\|_2^2},                                                                                                                                                                                                                                                                        \end{align*}
where $\tilde{\sigma}_k$, $k\in\{1, \dots, n\}$, are the HSVs of \eqref{structured_changed_system} with underlying Gramians solving \eqref{Gramians_strat1} as well as $u\in L^2_T$ and $v\in \mathbb R^d$. Moreover, $\tilde \gamma>0$ is a parameter charaterizing the operator $\tilde U$ and hence the Gramian $\tilde P$. The process $u_0$ taking values in $\mathbb R^{r_0\times d}$ is fixed and describes the underlying low-dimensional subspace of the auxiliary system \eqref{aux_var}.
\end{thm}
We found an a-priori error estimate for our dimension reduction procedure in Theorem \ref{error_bound_approach1}. In particular, the truncated HSVs $\tilde \sigma_{r+1}, \dots, \tilde \sigma_n$ deliver a nice criterion for the choice of a suitable $r$. Let us briefly recall two choices for $\tilde A$ and $\tilde N_i$ in \eqref{aux_var} mentioned in Subsection \ref{approach1}. First, we can fix $\tilde A=\tilde N_i=0$ yielding $r_0=d$, $V_0=X_0$ and consequently a constant $u_0$ being the identity of size $d\times d$. Therefore, we obtain $\mathbb E \int_0^T \|u_0(t)\|_2^2 dt=T$ meaning that the bound is increasing in the terminal time. This illustrates the drawback of exploiting a non-stable system \eqref{aux_var}, since the error bound explodes as $T\to \infty$. This is an indicator for a bad approximation quality on long time scales. Secondly,  we can use multiples of the identify, i.e., $\tilde A= -\alpha I$ and $\tilde N_i= \gamma_i I$, where $\gamma_i\in \mathbb R$. The parameter $\alpha$ can actually also be arbitrary, but for simplicity we set $\alpha>0$. This yields $r_0=d$, $V_0=X_0$ and $u_0(t)= \expn^{(-\alpha- 0.5 \sum_{i, j=1}^q \gamma_i \gamma_j k_{ij})t +\sum_{i=1}^q \gamma_i w_i(t)} I_{d\times d}$. We calculate that $\mathbb E\|u_0(t)\|_2^2=\expn^{\beta t}$, where $\beta= -2\alpha + \sum_{i, j=1}^q \gamma_i \gamma_j k_{ij}$. If $\beta=0$, we have the same effect as in the first case. Else, we see that $\mathbb E \int_0^T \|u_0(t)\|_2^2 dt= \frac{1}{\beta}(\expn^{\beta T}-1)$. Choosing $\gamma_i$, so that $\beta>0$ leads to an exponential growth in the error bound and therefore seems to be very non-optimal (unstable system \eqref{aux_var}). With $\beta<0$ (mean square asymptotically stable \eqref{aux_var}), a small error expression can be ensured and passing to the limit as $T\to \infty$ is possible in Theorem \ref{error_bound_approach1} if $u$ is square integrable on $\Omega\times[0, \infty)$, too.

We discuss an alternative to the reduced system  \eqref{stochstatenew_z_red2} in the next section.
  
\subsection{The approach according to Subsection \ref{approach2}}

                                                                                                                                                                              In contrast to Section \ref{approach1}, model reduction needs to applied twice when applying the alternative discussed below.  First of all, the control and secondly the initial state dynamics are tackled with a dimension reduction technique. We can simply write the output $y$ of the original system \eqref{original_system} as $y=(y-\tilde y) + \tilde y$, where $\tilde y$ is the output of \eqref{aux_var} with $\tilde A=A$ and $\tilde N_i=N_i$. Now, we can approximate the output $(y-\tilde y)$ of \eqref{sys_eq_z2} by $y_r$ associated to the reduced system \eqref{rom_x00}. This approach comes with an error bounded by an expression stated in Theorem \ref{thm_error_bound}. Moreover, we approximate $\tilde y$ by the output $\mathbf y_{\mathbf r}$ of the reduced dynamics in \eqref{rom_u0}. Here, the error is charaterized by Theorem \ref{thm_error_bound2}. In total, we obtain the following error estimate that we formulate in a theorem.                                                                                                                                                                              \begin{thm}   \label{thm_bound_ansat2}
Let $y$ be the quantity of interest in \eqref{original_system}  and $y_r$, $\mathbf y_{\mathbf r}$ the reduced order outputs of systems \eqref{rom_x00} and \eqref{rom_u0}, respectively. We assume that the Gramians $\mathbf P$ and $\mathbf Q$ associated to \eqref{original_system} with $u\equiv 0$ are positive definite. Then, we have                                                                                                                                                                         \begin{align*}                                                                                                                                                                                                                                                              \left\| y- y_{r}- \mathbf y_{\mathbf r}\right\|_{L^2_T}\leq 2\gamma^\frac{1}{2}\sum_{k=r+1}^n\sigma_k \|u\|_{L^2_T} + (\trace(\Theta_2 \mathcal W))^\frac{1}{2} \left\| v \right\|_2,                                                                                                                                                                                                                                                             \end{align*}
where $\sigma_i= \sqrt{\lambda_i(PQ)}$ are the HSVs of \eqref{original_system} with $x_0=0$ and $\Theta_2= \diag(\theta_{\mathbf r +1}, \dots, \theta_n)$ is the matrix of truncated HSVs $\theta_i= \sqrt{\lambda_i(\mathbf P \mathbf Q)}$ of \eqref{original_system} with $u\equiv 0$. The matrix $\mathcal W$ is a weight specified in Theorem \ref{thm_error_bound2},  $\gamma>0$ charaterizes the operator $U$, $u\in L^2_T$ and $v\in \mathbb R^d$.
\end{thm}
\begin{proof}
By the triangle inequality, we obtain that \begin{align*}
 \left\| y- y_{r}- \mathbf y_{\mathbf r}\right\|_{L^2_T}\leq \left\| (y-\tilde y)- y_{r}\right\|_{L^2_T}+\left\| \tilde y- \mathbf y_{\mathbf r}\right\|_{L^2_T},
 \end{align*}
 where $\tilde y$ is the output of \eqref{aux_var} with $\tilde A=A$ and $\tilde N_i=N_i$. The statement follows by applying Theorems \ref{thm_error_bound} and \ref{thm_error_bound2}.
 \end{proof}
We can also pass to the limit as $T\to \infty$ in Theorem \ref{thm_bound_ansat2} if $\lim_{T\to\infty}  \|u\|_{L^2_T}$ exists. Moreover, notice that the reduced dimensions $r$ and $\mathbf r$ may differ. This makes sense as the potential in reducing the control and the initial state dynamics can be very different. This is an advantage of this approach as, in this regard, it is more flexible than the one of Subsection \ref{approach1}. However, a second pair of  Gramians $\mathbf P$ and $\mathbf Q$ being solutions to  \eqref{gramP2} and \eqref{gramQ2} needs to be computed. At least this second pair is easier to compute than $P$ and $Q$, see Remark \ref{remark_ansatz2}.  Another drawback in comparison to the ansatz of Subsection \ref{approach1} is that the bound cannot be computed a-priori. On the other hand, Theorem \ref{thm_bound_ansat2} shows nicely how truncated HSVs of control and initial state dynamics are related to the error of the reduction.

\appendix
\section{Supporting lemma}
Below, we introduce a lemma that is based on Ito's product rule.
\begin{lem}\label{lemstochdiff}
Suppose that $A, B_1, \ldots, B_q$ are $\mathbb R^n$-valued $(\mathcal F_t)_{t\in[0, T]}$-adapted processes with $A$ being almost surely Lebesgue integrable and $B_i$ satisfying $\mathbb E\int_0^T \|B_i(t)\|_2^2 dt<\infty$. Let
 $w=\begin{bmatrix} w_1& \ldots & w_q\end{bmatrix}^\top$ be a Wiener  process with covariance matrix $K=(k_{ij})$. If $x$ is an Ito process given by \begin{align*}
 dx(t)=A(t) dt+ B(t)dw(t)=A(t) dt+ \sum_{i=1}^q B_i(t)dw_i(t),                                                                                                                                  \end{align*}
where $B=\begin{bmatrix} B_1& \ldots & B_q\end{bmatrix}$. Then, we have \begin{align*}
 \frac{d}{dt}\mathbb E\left[x(t)^\top x(t)\right] =2 \mathbb E\left[x(t)^\top A(t)\right] + \sum_{i, j=1}^q \mathbb E\left[B_i(t)^\top B_j(t)\right]k_{ij}.                                                                                                                              \end{align*}
\end{lem}
\begin{proof}
Ito's product rule yields $d\big(x(t)^\top x(t)\big) = d\big(x(t)^\top\big) x(t)+ x(t)^\top d\big(x(t)\big)+d\big(x(t)^\top\big) d\big(x(t)\big)$. We insert the representation of $x$ into this equation providing \begin{align*}
d\big(x(t)^\top x(t)\big)&= A(t)^\top x(t) dt+ (B(t)dw(t))^\top x(t) +  x(t)^\top  A(t) dt+ x(t)^\top B(t)dw(t) \\
&\quad+ \big(A(t)^\top dt+ \sum_{i=1}^q B_i(t)^\top dw_i(t)\big) \big(A(t) dt+\sum_{j=1}^q B_j(t) dw_j(t)\big).                                                                                                                                                                                                                    \end{align*}
 The result now follows by the rule  $dw_i(t)dw_j(t)=k_{ij}dt$ and by applying the expected value to the above identify exploiting the fact that the Ito integral has mean zero.
\end{proof}

\section*{Acknowledgments}
 MR is supported by the DFG via the individual grant ``Low-order approximations for large-scale problems arising in the context of high-dimensional
PDEs and spatially discretized SPDEs''-- project number 499366908.

\bibliographystyle{abbrv}
\bibliography{ref_nonlinear_mor}

\end{document}